\documentclass[11pt,twoside,reqno]{amsart}

\usepackage{version}         % include, exclude and comment
 \usepackage{color}

\usepackage[OT1]{fontenc}
\usepackage{type1cm}
\usepackage{amssymb}
\usepackage{mathrsfs}

\usepackage[left=2.3cm,top=3cm,right=2.3cm]{geometry}
\numberwithin{equation}{section}

\theoremstyle{plain}
\newtheorem{theorem}{Theorem}[section]
\newtheorem{maintheorem}{Theorem}

\newtheorem{proposition}[theorem]{Proposition}
\newtheorem{lemma}[theorem]{Lemma}

\theoremstyle{remark}
\newtheorem{remark}[theorem]{Remark}

\newtheorem*{ack}{Acknowledgement}

\theoremstyle{definition}

\newenvironment{labeledlist}[2][\unskip]
{ \renewcommand{\theenumi}{({#2}\arabic{enumi}{#1})}
  \renewcommand{\labelenumi}{({#2}\arabic{enumi}{#1})}
  \begin{enumerate} }
{ \end{enumerate} }

\newcommand{\LL}{\mathcal{L}}

\newcommand{\MM}{\mathcal{M}}

\newcommand{\R}{\mathbb{R}}

\newcommand{\N}{\mathbb{N}}

\newcommand{\iii}{\mathtt{i}}
\newcommand{\jjj}{\mathtt{j}}
\newcommand{\kkk}{\mathtt{k}}

\newcommand{\fii}{\varphi}

\newcommand{\la}{\langle}
\newcommand{\ra}{\rangle}

\DeclareMathOperator{\GL}{{\it GL_d}(\R)}
\DeclareMathOperator{\dimh}{dim_H}

\DeclareMathOperator{\dimaff}{dim_{aff}}

\DeclareMathOperator{\linspan}{span}

\renewcommand{\atop}[2]{\genfrac{}{}{0pt}{}{#1}{#2}}

\begin{document}

\title[Genericity of dimension drop on self-affine sets]{Genericity of dimension drop on self-affine sets}

\author{Antti K\"aenm\"aki}
\address{Department of Mathematics and Statistics \\
         P.O. Box 35 (MaD) \\
         FI-40014 University of Jyv\"askyl\"a \\
         Finland}
\email{antti.kaenmaki@jyu.fi}

\author{Bing Li$^*$}
\address{Department of Mathematics \\
         South China University of Technology \\
         Guangzhou, 510641 \\
         P.R.\ China}
\email{scbingli@scut.edu.cn}

\thanks{AK thanks the foreign expert visit program of the SCUT. BL acknowledges the support of NSFC 11671151 and Guangdong Natural Science Foundation 2014A030313230.}
\thanks{*Corresponding author.}
\date{\today}

\begin{abstract}
  We prove that generically, for a self-affine set in $\R^d$, removing one of the affine maps which defines the set results in a strict reduction of the Hausdorff dimension. This gives a partial positive answer to a folklore open question.
\end{abstract}

\maketitle

\section{Introduction and statement of main results}

Let $f_i \colon \R^d \to \R^d$ be a contractive invertible mapping for all $i \in \{ 1,\ldots,N \}$. Throughout of the paper, let $N$ be an integer such that $N \ge 2$. By Hutchinson \cite{Hutchinson1981}, there exists a unique nonempty compact set $E \subset \R^d$ satisfying
\begin{equation*}
  E = \bigcup_{i=1}^N f_i(E).
\end{equation*}
If the mappings $f_i$ are affine, then the set $E$ is called \emph{self-affine}. For every $\mathsf{v} = (v_1,\ldots,v_N) \in (\R^d)^N$ and $\mathsf{A} = (A_1,\ldots,A_N) \in GL_d(\R)^N$ with $\|A_i\| < 1$ for all $i \in \{ 1,\ldots,N \}$ we denote the self-affine set corresponding to $\mathsf{A}$ and $\mathsf{v}$ by
\begin{equation*}
  E_{\mathsf{A},\mathsf{v}} = \bigcup_{i=1}^N A_i(E_{\mathsf{A},\mathsf{v}})+v_i.
\end{equation*}
Then, by setting $\mathsf{A}' = (A_1,\ldots,A_{N-1}) \in GL_d(\R)^{N-1}$ and $\mathsf{v}' = (v_1,\ldots,v_{N-1}) \in (\R^d)^{N-1}$, the self-affine set
\begin{equation*}
  E_{\mathsf{A}',\mathsf{v}'} = \bigcup_{i=1}^{N-1} A_i(E_{\mathsf{A}',\mathsf{v}'})+v_i \subset E_{\mathsf{A},\mathsf{v}}
\end{equation*}
is obtained from $E_{\mathsf{A},\mathsf{v}}$ by removing one of the affine maps
which defines the set.

A folklore conjecture suggests that $\dimh(E_{\mathsf{A}',\mathsf{v}'}) < \dimh(E_{\mathsf{A},\mathsf{v}})$. There exist simple counter-examples showing that this cannot be the case for all self-affine sets; for example, see \cite[Example 9.3]{KaenmakiMorris2016}. Therefore the conjecture is about generic behavior. During recent years, the question has been propagated by Schmeling. It follows from Feng and K\"aenm\"aki \cite[\S 3]{FengKaenmaki2011} and Falconer \cite[Theorem 5.3]{Falconer1988} that the conjecture holds in $\R^2$ for Lebesgue almost every choice of translation vectors $\mathsf{v}$. Very recently, K\"aenm\"aki and Morris \cite[Theorem B]{KaenmakiMorris2016} solved the problem in dimension three. They showed that the conjecture holds in $\R^3$ again for Lebesgue almost every choice of translation vectors $\mathsf{v}$. We remark that, by Falconer and Miao \cite[Theorem 2.5]{FalconerMiao2007} and Falconer \cite[Theorem 5.3]{Falconer1988}, the conjecture holds in arbitrary dimension whenever the matrix tuple is upper triangular.

In this note, we present a simple proof to show that the conjecture holds in arbitrary dimension for a generic choice of the matrix tuple. This is a partial solution to the problem since one expects the conjecture to be true for all matrix tuples. At least, it is the case for rational dimensions:

\begin{maintheorem} \label{thm:schmeling-rational}
  Let $0 \le s \le d$ be rational and $\mathsf{A} = (A_1,\ldots,A_N) \in GL_d(\R)^N$ be such that $P_{\mathsf{A}}(\fii^s)=0$ and $\| A_i \| < 1/2$ for all $i \in \{ 1,\ldots,N \}$. Then
  \begin{equation*}
    \dimh(E_{\mathsf{A}',\mathsf{v}'}) < \dimh(E_{\mathsf{A},\mathsf{v}}) = s
  \end{equation*}
  for $\LL^{dN}$-almost all $\mathsf{v} \in \R^{dN}$.
\end{maintheorem}

The proof of Theorem \ref{thm:schmeling-rational} relies on understanding the structure and properties of equilibrium states obtained from the associated sub-additive dynamical system. Equilibrium state is a probability measure maximizing the dimension and it will be defined in \S \ref{sec:eq-state}.

For irrational dimensions, we verify the conjecture in an open and dense set of matrix tuples. The proof of this result is based on a property of matrix tuples called $s$-irreducibility. It is a sufficient condition for the uniqueness of the equilibrium state; see Theorem \ref{thm:unique-eq-new}.

\begin{maintheorem} \label{thm:schmeling-generic}
  For $\LL^{d^2N}$-almost every $\mathsf{A} = (A_1,\ldots,A_N) \in GL_d(\R)^N$ with $P_{\mathsf{A}}(\fii^d) \le 0$ and $\| A_i \| < 1/2$ for all $i \in \{ 1,\ldots,N \}$ it holds that
  \begin{equation*}
    \dimh(E_{\mathsf{A}',\mathsf{v}'}) < \dimh(E_{\mathsf{A},\mathsf{v}})
  \end{equation*}
  for $\LL^{dN}$-almost all $\mathsf{v} \in \R^{dN}$. In particular, the exceptional set of tuples $\mathsf{A} \in GL_d(\R)^N$ for which the conclusion does not hold is contained in a finite union of $(d^2N-1)$-dimensional algebraic varieties and thus, has Hausdorff dimension at most $d^2N-1$. 
\end{maintheorem}

In the above theorems, $P_{\mathsf{A}}$ is the singular value pressure of $\mathsf{A}$ defined in \eqref{eq:topo_def}.

\begin{remark}
  Since $P_{\mathsf{A}}(\fii^s)$ is continuous as a function of $s$ and also, by Feng and Shmerkin \cite[Theorem 1.2]{FengShmerkin2014}, as a function of $\mathsf{A}$, it is tempting to try to solve the full conjecture just by taking a limit. Unfortunately, this approach does not seem to work -- at least not without any further modifications.
\end{remark}

\section{Preliminaries}

\subsection{Set of infinite words}
Fix $N \in \N$ such that $N \ge 2$ and equip the set of all infinite words $\Sigma = \{ 1,\ldots,N \}^\N$ with the usual ultrametric: the distance between two different words is defined to be $2^{-n}$, where $n$ is the first place at which the words differ. It is straightforward to see that $\Sigma$ is compact. The \emph{left shift} is a continuous map $\sigma \colon \Sigma \to \Sigma$ defined by setting $\sigma(\iii) = i_2 i_3 \cdots$ for all $\iii = i_1 i_2 \cdots \in \Sigma$.

Let $\Sigma_*$ be the free monoid on $\{ 1,\ldots,N \}$. The concatenation of two words $\iii \in \Sigma_*$ and $\jjj \in \Sigma_* \cup \Sigma$ is denoted by $\iii\jjj \in \Sigma_* \cup \Sigma$. The set $\Sigma_*$ is the set of all finite words $\{ \varnothing \} \cup \bigcup_{n \in \N} \Sigma_n$, where $\Sigma_n = \{ 1,\ldots,N \}^n$ for all $n \in \N$ and $\varnothing$ satisfies $\varnothing\iii = \iii\varnothing = \iii$ for all $\iii \in \Sigma_*$. For notational convenience, we set $\Sigma_0 = \{ \varnothing \}$. The word $i_2 \cdots i_n \in \Sigma_{n-1}$ is denoted by $\sigma(\iii)$ for all $n \in \N$ and $\iii = i_1 \cdots i_n \in \Sigma_n$.

The length of $\iii \in \Sigma_* \cup \Sigma$ is denoted by $|\iii|$. If $\iii \in \Sigma_*$, then we set
$
  [\iii] = \{ \iii\jjj \in \Sigma : \jjj \in \Sigma \}
$
and call it a \emph{cylinder set}. If $\jjj \in \Sigma_* \cup \Sigma$ and $1 \le n < |\jjj|$, we define $\jjj|_n$ to be the unique word $\iii \in \Sigma_n$ for which $\jjj \in [\iii]$. If $\jjj \in \Sigma_*$ and $n \ge |\jjj|$, then $\jjj|_n = \jjj$. % We also set $\iii^- = \iii|_{|\iii|-1}$ for all $\iii \in \Sigma_* \setminus \{ \varnothing \}$.

\subsection{Multilinear algebra} \label{sec:multilinear}
We recall some basic facts about the exterior algebra and tensor products.
Let $\{ e_1,\ldots,e_d \}$ be the standard orthonormal basis of $\R^d$ and define
\begin{equation*}
  \wedge^k \R^d = \linspan\{ e_{i_1} \wedge \cdots \wedge e_{i_k} : 1 \le i_1 < \cdots < i_k \le d \}
\end{equation*}
for all $k \in \{ 1,\ldots,d \}$ with the convention that $\wedge^0\R^d=\mathbb{R}$. Recall that the wedge product $\wedge \colon \wedge^k\R^d \times \wedge^j\R^d \to \wedge^{k+j}\R^d$ is an associative and bilinear operator, anticommutative on the elements of $\R^d$.
For each $v \in \wedge^k\R^d$ and $1 \le i_1 < \cdots < i_k \le d$ there exists a real number $v_{i_1 \cdots i_k}$ so that
\begin{equation} \label{eq:coefficients}
  v = \sum_{1 \le i_1 < \cdots < i_k \le d} v_{i_1 \cdots i_k} e_{i_1} \wedge \cdots \wedge e_{i_k}
\end{equation}
Moreover, if $v_j = (v_j^1,\ldots,v_j^d) \in \R^d$ for all $j \in \{ 1,\ldots,k \}$, then
\begin{equation*} %\label{eq:decomposable}
  v_1 \wedge \cdots \wedge v_k = \sum_{1 \le i_1 < \cdots < i_k \le d} \det
  \begin{pmatrix}
    v_1^{i_1} & \cdots & v_1^{i_k} \\
    \vdots & \ddots & \vdots \\
    v_k^{i_1} & \cdots & v_k^{i_k}
  \end{pmatrix}
  e_{i_1} \wedge \cdots \wedge e_{i_k}.
\end{equation*}
If $v \in \wedge^k\R^d$ can be expressed as a wedge product of $k$ vectors of $\R^d$, then $v$ is said to be \emph{decomposable}. Observe that e.g.\ $e_1 \wedge e_2 + e_3 \wedge e_4 \in \wedge^2\R^4$ is not decomposable.
%If $v \in \wedge^k\R^d$ is decomposable, then the coefficients $v_{i_1 \cdots i_k}$ in \eqref{eq:coefficients} are just the corresponding minors.

The group of $d \times d$ invertible matrices of real numbers is denoted by $GL_d(\R)$. This space has a topology induced from $\R^{d^2}$. If $A \in GL_d(\R)$, we define an invertible linear map $A^{\wedge k} \colon \wedge^k\R^d \to \wedge^k\R^d$ by setting
\[
(A^{\wedge k})(e_{i_1} \wedge \cdots \wedge e_{i_k}) = Ae_{i_1} \wedge \cdots \wedge Ae_{i_k}
\]
and extending by linearity. Observe that $A^{\wedge k}$ can be represented by a $\binom{d}{k} \times \binom{d}{k}$ matrix whose entries are the $k \times k$ minors of $A$. Using this and standard properties of determinants, it may be shown that
\begin{equation} \label{eq:morphism}
  (AB)^{\wedge k} = (A^{\wedge k})(B^{\wedge k}),
\end{equation}
i.e.\ $A \mapsto A^{\wedge k}$ is a morphism between the corresponding multiplicative linear groups. Furthermore, if $\alpha_1(A) \ge \cdots \ge \alpha_d(A) > 0$ are the singular values of $A$, that is, the square roots of the eigenvalues of the positive definite matrix $A^TA$, where $A^T$ is the transpose of $A$, then the products $\alpha_{i_1}(A) \cdots \alpha_{i_k}(A)$ are the singular values of $A^{\wedge k}$, for each $1 \le i_1 < \cdots < i_k \le d$. 

The inner product on $\wedge^k\R^d$ is defined by setting
\begin{equation*}
  \langle v,w \rangle_k = \sum_{1 \le i_1 < \cdots < i_k \le d} v_{i_1 \cdots i_k} w_{i_1 \cdots i_k}
\end{equation*}
for all $v,w \in \wedge^k\R^d$, where $v_{i_1 \cdots i_k}$ and $w_{i_1 \cdots i_k}$ are the coefficients of the corresponding linear combinations; see \eqref{eq:coefficients}.
The norm is defined by setting $|v|_k = \langle v,v \rangle_k^{1/2}$ for all $v \in \wedge^k\R^d$. It follows that $|v_1 \wedge \cdots \wedge v_k|_k$ is the $k$-dimensional volume of the parallelepiped with the vectors $v_1,\ldots,v_k$ as sides. The operator norm of the induced linear mapping $A^{\wedge k}$ is
\begin{equation} \label{eq:operator_norm_for_wedge}
  \| A^{\wedge k} \|_k = \max\{ |A^{\wedge k}v|_k : |v|_k = 1 \} = \alpha_1(A) \cdots \alpha_k(A).
\end{equation}

The \emph{tensor product} of two inner product spaces $V$ and $W$ over $\R$ is the inner product space $V \otimes W$. Its elements are equivalence classes of formal sums of vectors in $V \times W$ with coefficients in $\R$ under a natural equivalence relation.
If $v \in V$ and $w \in W$, then the equivalence class of $(v,w)$ is denoted by $v \otimes w$, which is called the \emph{tensor product} of $v$ with $w$. An element of $V \otimes W$ is called \emph{decomposable} if it can be expressed as a tensor product of two vectors in $V$ and $W$. Observe that if $v_1,v_2 \in V$ and $w_1,w_2 \in W$ are both linearly independent, then $v_1 \otimes w_1 + v_2 \otimes w_2$ is not decomposable. If $\{ e_i \}_i$ is a basis for $V$ and $\{ e_j' \}_j$ is a basis for $W$, then $\{ e_i \otimes e_j' \}_{i,j}$ is a basis for $V \otimes W$. The dimension of the tensor product space therefore is the product of dimensions of the original spaces.

The inner product on $V \otimes W$ is defined by
\begin{equation*}
  \langle v_1 \otimes w_1, v_2 \otimes w_2 \rangle = \langle v_1 , v_2 \rangle \langle w_1, w_2 \rangle
\end{equation*}
for decomposable elements of $V \otimes W$ and by bilinear extension for general elements of $V \otimes W$. If $T \colon V \to V$ and $U \colon W \to W$ are linear maps, then, by setting
\begin{equation*}
  T \otimes U(v \otimes w) = T(v) \otimes U(w)
\end{equation*}
for decomposable elements and extending by linearity, defines a linear map $T \otimes U \colon V \otimes W \to V \otimes W$ which is called the \emph{tensor product} of $T$ and $U$. If the linear maps $T$ and $U$ are considered to be matrices, then the matrix describing the tensor product $T \otimes U$ is the usual \emph{Kronecker product} of the two matrices. Since the norm is defined by $|v| = \langle v,v \rangle^{1/2}$ for all $v \in V \otimes W$ the operator norm of the tensor product of $T$ and $U$ is the product of the operator norms of $T$ and $U$, i.e.\
\begin{equation} \label{eq:tensor_norm_product}
  \| T \otimes U \| = \| T \| \| U \|.
\end{equation}

\subsection{Irreducibility}
Let $\mathcal{A}$ be a set of matrices in $GL_d(\R)$. We say that $\mathcal{A}$ is \emph{irreducible} if there is no proper nontrivial linear subspace $V$ of $\R^d$ such that $A(V) \subset V$ for all $A \in \mathcal{A}$; otherwise $\mathcal{A}$ is called \emph{reducible}.
A tuple $\mathsf{A} = (A_1,\ldots,A_N) \in GL_d(\R)^N$ is \emph{irreducible}
if the corresponding set $\{ A_1,\ldots,A_N \}$ is irreducible.
If $\mathsf{A}^{\wedge k} = (A^{\wedge k}_1, \ldots, A^{\wedge k}_N)$ is irreducible
for some $k \in \{ 0,\ldots,d \}$, then we say that $\mathsf{A}$ is \emph{$k$-irreducible}.
For each $n \in \N$ and $\iii = i_1 \cdots i_n \in \Sigma_n$ we write $A_\iii = A_{i_1} \cdots A_{i_n} \in GL_d(\R)$.

\begin{lemma} \label{thm:irr_equiv1}
  If $\mathsf{A} = (A_1,\ldots,A_N) \in GL_d(\R)^N$, then the following conditions are equivalent:
  \begin{enumerate}
    \item The tuple $\mathsf{A}$ is irreducible.
    \item For every $0\neq v,w \in \R^d$ there is $\iii \in \Sigma_*$ such that $\langle v,A_\iii w \rangle \ne 0$.
    \item For every $0 \ne w \in \R^d$ it holds that $\linspan(\{ A_\iii w : \iii \in \Sigma_* \}) = \R^d$.
    \item The set $\{ A_\iii : \iii \in \Sigma_* \}$ is irreducible.
  \end{enumerate}
\end{lemma}

\begin{proof}
  Although the proof is similar to that of \cite[Lemma 2.6]{FalconerSloan2009}, we give it here for the convenience of the reader. To show that (1) $\Rightarrow$ (2), suppose that the condition (2) is not satisfied. Then there exist $0\neq v,w \in \R^d$ such that $\langle v,A_\iii w \rangle = 0$ for all $\iii \in \Sigma_*$. This means that $v$ is orthogonal to the non-trivial proper linear subspace $V = \linspan\{ A_\iii w : \iii \in \Sigma_* \}$ of $\R^d$. Since trivially $A_i(V) \subset V$ for all $i \in \{ 1,\ldots,N \}$ we have shown that (1) does not hold, and thus finished the proof of the implication.

  To show that (2) $\Rightarrow$ (3), suppose to the contrary that there exists $0 \ne w \in \R^d$ such that $V=\linspan(\{ A_\iii w : \iii \in \Sigma_* \})$ is non-trivial proper subspace of $\R^d$. Thus there exists $0 \ne v \in V^\perp$. Since now $\langle v,A_\iii w \rangle = 0$ for all $\iii \in \Sigma_*$ we have shown that (2) does not hold, and thus finished the proof of the implication.

  To show that (3) $\Rightarrow$ (4), assume contrarily that there exists a non-trivial proper linear subspace $V$ of $\R^d$ such that $A_\iii(V) \subset V$ for all $\iii \in \Sigma_*$. Let $0 \ne w \in V$. Since $A_\iii w \subset V$ for all $\iii \in \Sigma_*$ we have shown that $\linspan(\{ A_\iii w : \iii \in \Sigma_* \})$ is a non-trivial proper linear subspace of $\R^d$. Thus (3) does not hold and we have finished the proof of this implication.
  
  Since the implication (4) $\Rightarrow$ (1) is trivial we have finished the whole proof.
\end{proof}

Let $k \in \{ 0,\ldots,d-1 \}$ and $k<s<k+1$. We say that $\mathsf{A}$ is \emph{$s$-irreducible} if for every $v_1,w_1 \in \wedge^k\R^d$ and $v_2,w_2 \in \wedge^{k+1}\R^d$ there is $\iii \in \Sigma_*$ such that
\begin{equation*}
  \langle v_1,A^{\wedge k}_\iii w_1 \rangle_k \ne 0 \quad \text{and} \quad \langle v_2,A^{\wedge(k+1)}_\iii w_2 \rangle_{k+1} \ne 0.
\end{equation*}
Observe that, by Lemma \ref{thm:irr_equiv1}, if $\mathsf{A}$ is $s$-irreducible, then it is $k$-irreducible and $(k+1)$-irreducible. It follows from \cite[Example 9.2 and Lemma 3.3]{KaenmakiMorris2016} and Theorem \ref{thm:unique-eq-new} that the converse is not true. We emphasize that if $\mathsf{A}$ is $s$-irreducible for some $k<s<k+1$, then it is $s$-irreducible for all $k<s<k+1$. Note that $s$-irreducibility for $0<s<1$ is just $1$-irreducibility. The following lemma gives a connection between the $s$-irreducibility and irreducibility.

\begin{lemma} \label{thm:not-s-irr}
  Let $k \in \{ 0,\ldots,d-1 \}$ and $k<s<k+1$. If $\mathsf{A} = (A_1,\ldots,A_N) \in \GL^N$ is not $s$-irreducible, then $\mathsf{A}' = (A_1^{\wedge k} \otimes A_1^{\wedge(k+1)}, \ldots, A_N^{\wedge k} \otimes A_N^{\wedge(k+1)})$ is not irreducible.
\end{lemma}

\begin{proof}
  Write $A_i' = A_i^{\wedge k} \otimes A_i^{\wedge(k+1)}$ for all $i$ and observe that $A_i'$ is an invertible linear map acting on $\wedge^k\R^d \otimes \wedge^{k+1}\R^d$. Assume to the contrary that $\mathsf{A}'$ is irreducible. Then, by Lemma \ref{thm:irr_equiv1}, for every $v,w \in \wedge^k\R^d \otimes \wedge^{k+1}\R^d$ there is $\iii \in \Sigma_*$ such that $\langle v,A_\iii' w \rangle \ne 0$. Since, in particular, this holds for decomposable elements, we see that for every $(v_1,v_2), (w_1,w_2) \in \wedge^k\R^d \times \wedge^{k+1}\R^d$ there is $\iii \in \Sigma_*$ such that
  \begin{equation*}
    \langle v_1, A_\iii^{\wedge k}w_1 \rangle_k \, \langle v_2, A_\iii^{\wedge(k+1)}w_2 \rangle_{k+1} = \langle v_1 \otimes v_2, A_\iii'(w_1 \otimes w_2) \rangle \ne 0.
  \end{equation*}
  Therefore, $\mathsf{A}$ is $s$-irreducible.
\end{proof}

\begin{remark} \label{rem:morris}
  We remark that the above lemma is far from being a characterization. For example, the set $\{ A \otimes A^{\wedge 2} : A \in GL_3(\R) \}$ is not irreducible. To see this, observe that the $1$-dimensional subspace spanned by $v = e_1 \otimes (e_2 \wedge e_3) + e_2 \otimes (e_3 \wedge e_1) + e_3 \otimes (e_1 \wedge e_2)$ is invariant for all $A \otimes A^{\wedge 2}$. Indeed, it is straightforward to see that $(A \otimes A^{\wedge 2})v = \det(A)v$ for all upper and lower triangular $A \in GL_3(\R)$ and hence, by LU decomposition, for all $A \in GL_3(\R)$.
\end{remark}

A slightly modified version of the condition $C(s)$ introduced and used by Falconer and Sloan \cite{FalconerSloan2009} implies $s$-irreducibility. We say that the set $\mathcal{A}$ of matrices in $GL_d(\R)$ satisfies the \emph{condition $C(s)$} if for every $v_1,w_1 \in \wedge^k \R^d$ and $v_2,w_2 \in \wedge^{k+1} \R^d$ there is $A \in \mathcal{A}$ such that
\begin{equation} \label{eq:cs-def}
  \la v_1,A^{\wedge k} w_1 \ra_k \ne 0 \quad \text{and} \quad \la v_2,A^{\wedge(k+1)} w_2 \ra_{k+1} \ne 0.
\end{equation}
A tuple $\mathsf{A} = (A_1\ldots,A_N) \in GL_d(\R)^N$ satisfies the condition $C(s)$ if the corresponding set $\{ A_1\ldots,A_N \}$ satisfies the condition $C(s)$. The condition $C(k)$ holds if the left-hand side equation in \eqref{eq:cs-def} is satisfied. We remark that the condition $C(s)$ in \cite{FalconerSloan2009} is slightly weaker: instead of arbitrary $v_2,w_2 \in \wedge^{k+1} \R^d$ they require \eqref{eq:cs-def} only for vectors of the form $v_2 = v_1 \wedge v$ and $w_2 = w_1 \wedge w$. The following lemma follows immediately from the definitions.

\begin{lemma} \label{thm:CS-cond}
  Let $k \in \{ 0,\ldots,d-1 \}$ and $k \le s < k+1$. Then $\mathsf{A} = (A_1,\ldots,A_N) \in GL_d(\R)^N$ is $s$-irreducible if and only if $\{ A_\iii : \iii \in \Sigma_* \}$ satisfies the condition $C(s)$. In particular, if $\mathsf{A}$ satisfies the condition $C(s)$, then $\mathsf{A}$ is $s$-irreducible.
\end{lemma}

\subsection{Singular value function} \label{sec:svf-pressure}

Let $k \in \{ 0,\ldots,d-1 \}$ and $k \le s < k+1$. We define the \emph{singular value function} to be
\begin{equation*}
  \fii^s(A) = \| A^{\wedge k} \|_{k}^{k+1-s} \; \| A^{\wedge(k+1)} \|_{k+1}^{s-k} = \alpha_1(A) \cdots \alpha_{k}(A) \alpha_{k+1}(A)^{s-k}
\end{equation*}
for all $A \in GL_d(\R)$ with the convention that $\|A^{\wedge 0}\|_0 = 1$. Observe that \eqref{eq:morphism} and the submultiplicativity of the operator norm imply
\begin{equation} \label{eq:cylinder2}
\begin{split}
  \fii^s(AB) &= \| (AB)^{\wedge k} \|_{k}^{k+1-s} \; \| (AB)^{\wedge(k+1)} \|_{k+1}^{s-k} \\ &\le \| A^{\wedge k} \|_{k}^{k+1-s} \; \| B^{\wedge k} \|_{k}^{k+1-s} \; \| A^{\wedge(k+1)} \|_{k+1}^{s-k} \; \| B^{\wedge(k+1)} \|_{k+1}^{s-k} = \fii^s(A) \fii^s(B)
\end{split}
\end{equation}
for all $A,B \in GL_d(\R)$. When $s \ge d$, we set $\fii^s(A) = |\det(A)|^{s/d}$ for completeness.

For a given $\mathsf{A} = (A_1,\ldots,A_N) \in GL_d(\R)^N$ we define
\begin{equation} \label{eq:topo_def}
  P_{\mathsf{A}}(\fii^s) = \lim_{n \to \infty} \tfrac{1}{n} \log\sum_{\iii \in \Sigma_n} \fii^s(A_\iii)
\end{equation}
and call it the \emph{singular value pressure} of $\mathsf{A}$. The limit above exists by the standard theory of subadditive sequences. It is easy to see that, as a function of $s$, the singular value pressure is continuous, strictly decreasing, and convex between any two consecutive integers. Furthermore, since $P_{\mathsf{A}}(\fii^0)=\log N > 0$ and $\lim_{s \to \infty} P_{\mathsf{A}}(\fii^s) = -\infty$ there exists unique $s \ge 0$ for which $P_{\mathsf{A}}(\fii^s)=0$. The minimum of $d$ and this $s$ is called the \emph{affinity dimension} of $\mathsf{A}$ and is denoted by $\dimaff(\mathsf{A})$.

\subsection{Equilibrium states} \label{sec:eq-state}
We denote the collection of all Borel probability measures on $\Sigma$ by $\MM(\Sigma)$, and endow it with the weak$^*$ topology. We say that $\mu \in \MM(\Sigma)$ is \emph{fully supported} if $\mu([\iii]) > 0$ for all $\iii \in \Sigma_*$. Let
\[
  \MM_\sigma(\Sigma) = \{ \mu \in \MM(\Sigma) : \mu \text{ is $\sigma$-invariant} \},
\]
where \emph{$\sigma$-invariance} of $\mu$ means that
$\mu([\iii]) = \mu(\sigma^{-1}([\iii])) = \sum_{i=1}^N \mu([i\iii])$
for all $\iii \in \Sigma_*$. Observe that if $\mu \in \MM_\sigma(\Sigma)$, then $\mu(A)=\mu(\sigma^{-1}(A))$ for all Borel sets $A \subset \Sigma$.
We say that $\mu$ is \emph{ergodic} if $\mu(A) = 0$ or $\mu(A) = 1$ for every Borel set $A \subset \Sigma$ with $A = \sigma^{-1}(A)$. Recall that the set $\MM_\sigma(\Sigma)$ is compact and convex with ergodic measures as its extreme points.

If $\mu \in \MM_\sigma(\Sigma)$, then we define the \emph{entropy} $h$ of $\mu$ by setting
\begin{equation*}
  h(\mu) = \lim_{n \to \infty} \tfrac{1}{n} \sum_{\iii \in \Sigma_n} -\mu([\iii]) \log\mu([\iii]).
\end{equation*}
In addition, if $0 \le s \le d$ and $\mathsf{A} = (A_1,\ldots,A_N) \in GL_d(\R)^N$, then we define
\begin{equation*}
  \lambda_{\mathsf{A}}(\fii^s,\mu) = \lim_{n \to \infty} \tfrac{1}{n} \sum_{\iii \in \Sigma_n} \mu([\iii]) \log\fii^s(A_\iii).
\end{equation*}
Recalling \eqref{eq:cylinder2} and the fact that $\mu$ is invariant, the limits above exist and equal the infimums by the standard theory of subadditive sequences.

An application of Jensen's inequality yields
$P_{\mathsf{A}}(\fii^s) \ge h(\mu) + \lambda_{\mathsf{A}}(\fii^s,\mu)$
for all $\mu \in \MM_\sigma(\Sigma)$ and $s \ge 0$. A measure $\mu \in \MM_\sigma(\Sigma)$ is called an \emph{$\varphi^s$-equilibrium state} of $\mathsf{A}$ if it satisfies the following variational principle:
\begin{equation*}
  P_{\mathsf{A}}(\fii^s) = h(\mu) + \lambda_{\mathsf{A}}(\fii^s,\mu).
\end{equation*}
K\"{a}enm\"{a}ki \cite[Theorems 2.6 and 4.1]{Kaenmaki2004} proved that for each $\mathsf{A} \in GL_d(\R)^N$ and $s \ge 0$ there exists an ergodic $\fii^s$-equilibrium state of $\mathsf{A}$; see also \cite[Theorem 3.3]{KaenmakiVilppolainen2010}. The example of K\"aenm\"aki and Vilppolainen \cite[Example 6.2]{KaenmakiVilppolainen2010} shows that such an equilibrium state is not necessarily unique.

The following theorem shows that $s$-irreducibility is a sufficient condition for the uniqueness of the $\fii^s$-equilibrium state. Its proof follows by applying \cite[Lemma 2.5]{KaenmakiMorris2016} in \cite[Theorem A]{KaenmakiReeve2014}.

\begin{theorem} \label{thm:unique-eq-new}
  Let $k \in \{ 0,\ldots,d-1 \}$ and $k < t < k+1$. If $\mathsf{A} = (A_1,\ldots,A_N) \in GL_d(\R)^N$ is $t$-irreducible, then for every $k<s<k+1$ there exists a unique $\fii^s$-equilibrium state $\mu$ of $\mathsf{A}$ and it satisfies the following condition: there exists $C \ge 1$ depending only on $\mathsf{A}$ and $s$ such that
  \begin{equation*}
    C^{-1}e^{-nP_{\mathsf{A}}(\fii^s)}\fii^s(A_\iii) \le \mu([\iii]) \le Ce^{-nP_{\mathsf{A}}(\fii^s)}\fii^s(A_\iii)
  \end{equation*}
  for all $\iii \in \Sigma_*$.
\end{theorem}

Similarly as in \eqref{eq:topo_def}, given $\mathsf{A} = (A_1,\ldots,A_N) \in GL_d(\R)^N$ and $s \ge 0$, we define
\begin{equation*}
  P_{\mathsf{A}}(\|\cdot\|^s) = \lim_{n \to \infty} \tfrac{1}{n} \log\sum_{\iii \in \Sigma_n} \| A_\iii \|^s
\end{equation*}
and call it the \emph{norm pressure} of $\mathsf{A}$. Note that $P_{\mathsf{A}}(\|\cdot\|^s) = P_{\mathsf{A}}(\fii^s)$ for all $0 \le s \le 1$. If $\mu \in \MM_\sigma(\Sigma)$, then we also set
\begin{equation*}
  \lambda_{\mathsf{A}}(\|\cdot\|^s,\mu) = \lim_{n \to \infty} \tfrac{1}{n} \sum_{\iii \in \Sigma_n} \mu([\iii]) \log\|A_\iii\|^s.
\end{equation*}
It follows that $P_{\mathsf{A}}(\|\cdot\|^s) \ge h(\mu) + \lambda_{\mathsf{A}}(\|\cdot\|^s,\mu)$ for all $\mu \in \MM_\sigma(\Sigma)$ and $s \ge 0$. A measure $\mu \in \MM_\sigma(\Sigma)$ is called a \emph{$\|\cdot\|^s$-equilibrium state} of $\mathsf{A}$ if
\begin{equation*}
  P_{\mathsf{A}}(\|\cdot\|^s) = h(\mu) + \lambda_{\mathsf{A}}(\|\cdot\|^s,\mu).
\end{equation*}
The following theorem is proved by Feng and K\"aenm\"aki \cite[Theorem 1.7]{FengKaenmaki2011}.

\begin{theorem} \label{thm:FeKa}
  If $0 \le s \le d$ and $\mathsf{A} \in GL_d(\R)^N$, then there exist at most $d$ distinct ergodic $\|\cdot\|^s$-equilibrium states of $\mathsf{A}$ and they are all fully supported. Furthermore, if $\mathsf{A}$ is irreducible, then the equilibrium state is unique.
\end{theorem}

As became apparent in \cite{KaenmakiMorris2016}, this result is useful also in the study of $\fii^s$-equilibrium states.

\section{Proofs of the main results}

To prove Theorems \ref{thm:schmeling-rational} and \ref{thm:schmeling-generic}, we rely on the following proposition. It is proved in \cite[Proposition 8.1]{KaenmakiMorris2016} and its proof is a simple consequence of the variational principle.

\begin{proposition} \label{thm:pressure-drop}
  Let $0 \le s \le d$ and $\mathsf{A} = (A_1,\ldots,A_N) \in GL_d(\R)^N$. If all the $\fii^s$-equilibrium states of $\mathsf{A}$ are fully supported and $\mathsf{A}' = (A_1,\ldots,A_{N-1}) \in GL_d(\R)^{N-1}$, then $P_{\mathsf{A}'}(\fii^s) < P_{\mathsf{A}}(\fii^s)$. Moreover, if $P_{\mathsf{A}}(\fii^d) \le 0$, then $\dimaff(\mathsf{A}') < \dimaff(\mathsf{A})$.
\end{proposition}

Let us first focus on Theorem \ref{thm:schmeling-rational}.

\begin{theorem} \label{thm:rational-eq}
  If $0 \le s \le d$ is rational and $\mathsf{A} \in GL_d(\R)^N$, then there exist at most finitely many distinct ergodic $\fii^s$-equilibrium states of $\mathsf{A}$ and they are all fully supported.
\end{theorem}

\begin{proof}
  Following \cite[\S 5]{Morris2015}, we will express the singular value function as a norm of a tensor product. The assumption that $s$ is rational is essential here. Note that if $s$ is an integer, then the claim follows immediately from \cite[\S 3]{FengKaenmaki2011}. Let $p,q \in \N$ be such that $s=p/q$ and let $k \in \{ 0,\ldots,d-1 \}$ be such that $k<p/q<k+1$.
  If $A \in \GL$, then
  \begin{equation} \label{eq:rational-matrix}
    A' = (\otimes_{i=1}^{(k+1)q-p} A^{\wedge k}) \otimes (\otimes_{j=1}^{p-k q} A^{\wedge(k+1)})
  \end{equation}
  is an invertible linear map acting on the $\binom{d}{k}^{(k+1)q-p} \binom{d}{k+1}^{p-k q}$-dimensional inner product space $(\otimes_{i=1}^{(k+1)q-p} \wedge^k\R^d) \otimes (\otimes_{j=1}^{p-k q} \wedge^{k+1}\R^d)$. Observe that \eqref{eq:tensor_norm_product} and \eqref{eq:operator_norm_for_wedge} give
  \begin{align*}
    \| A' \| &= \| A^{\wedge k} \|_k^{(k+1)q-p} \| A^{\wedge(k+1)} \|_{k+1}^{p-k q} = (\alpha_1(A) \cdots \alpha_k(A))^{(k+1)q-p} (\alpha_1(A) \cdots \alpha_{k+1}(A))^{p-k q} \\
    &= (\alpha_1(A) \cdots \alpha_k(A))^q \alpha_{k+1}(A)^{p-\ell q} = \fii^s(A)^q
  \end{align*}
  and hence
  \begin{equation*}
    \fii^{s}(A) = \| A' \|^{1/q}.
  \end{equation*}
  The set of $\fii^s$-equilibrium states of $\mathsf{A}$ is therefore precisely the set of $\| \cdot \|^{1/q}$-equilibrium states of $\mathsf{A}'$, where $\mathsf{A}' = (A_1',\ldots,A_N') \in GL_{d'}(\R)^N$, $d' = \binom{d}{k}^{(k+1)q-p} \binom{d}{k+1}^{p-k q}$, and each $A_i'$ is of the form \eqref{eq:rational-matrix}. Thus, by Theorem \ref{thm:FeKa}, there exist at most $d'$ distinct ergodic $\fii^s$-equilibrium states of $\mathsf{A}$ and they are all fully supported.
\end{proof}

\begin{remark}
  (1) Unfortunately the proof of the previous theorem does not show that the number of $\fii^s$-equilibrium states is bounded over $s$.

  (2) A small modification of Lemma \ref{thm:not-s-irr} shows that if $\mathsf{A} = (A_1,\ldots,A_N) \in GL_d(\R)^N$ is not $s$-irreducible, then $\mathsf{A}' = (A_1',\ldots,A_N')$, where each $A_i'$ is of the form \eqref{eq:rational-matrix}, is not irreducible.
\end{remark}

\begin{proof}[Proof of Theorem \ref{thm:schmeling-rational}]
  If $\mathsf{A}' = (A_1,\ldots,A_{N-1}) \in GL_d(\R)^{N-1}$, then Theorem \ref{thm:rational-eq} and Proposition \ref{thm:pressure-drop} imply that $\dimaff(\mathsf{A}') < \dimaff(\mathsf{A})$. Therefore, the result of Falconer \cite[Theorem 5.3]{Falconer1988} finishes the proof.
\end{proof}

Let us then turn to Theorem \ref{thm:schmeling-generic}. We will first give a sufficient and checkable condition for $s$-irreducibility. We say that $(A_1,A_2) \in GL_d(\R)^2$ satisfies the \emph{eigenvalue condition} if both matrices have $d$ distinct eigenvalues (real or complex) and the following two conditions are satisfied:
\begin{labeledlist}{E}
  \item The eigenvalues of $A_1$ and $A_2$, respectively denoted by $\lambda_1,\ldots,\lambda_d$ and $\lambda_1',\ldots,\lambda_d'$, satisfy
  \begin{equation*}
    \lambda_{i_1} \cdots \lambda_{i_k} \ne \lambda_{j_1} \cdots \lambda_{j_k} \quad \text{and} \quad \lambda_{i_1}' \cdots \lambda_{i_k}' \ne \lambda_{j_1}' \cdots \lambda_{j_k}'
  \end{equation*}
  for all pairs $(i_1,\ldots,i_k) \ne (j_1,\ldots,j_k)$ and for all $k \in \{ 1,\ldots,d \}$. \label{E1}
  \item If the eigenvectors of $A_1$ and $A_2$ corresponding to the eigenvalues are $e_1,\ldots,e_d$ and $e_1',\ldots,e_d'$, respectively, and $X \in GL_d(\R)$ is the change of basis matrix for which $Xe_i'=e_i$ for all $i \in \{ 1,\ldots,d \}$, then all the minors of $X$ are non-zero. \label{E2}
\end{labeledlist}
In \ref{E2}, if an eigenvalue $\lambda$ is complex, then its complex conjugate $\overline{\lambda}$ is also an eigenvalue. In this case, we choose eigenvectors to be any two linearly independent vectors spanning the invariant plane corresponding to $\lambda$ and $\overline{\lambda}$.

Furthermore, we say that $\mathsf{A} = (A_1,\ldots,A_N) \in GL_d(\R)^N$ satisfies the eigenvalue condition if there exist $i,j \in \{ 1,\ldots,N \}$ such that $i \ne j$ and $(A_i,A_j)$ satisfies the eigenvalue condition. The following proposition shows that the eigenvalue condition is generic.

\begin{proposition} \label{thm:eigenvalue-cond}
  The set $\{ \mathsf{A} \in GL_d(\R)^N : \mathsf{A} \text{ satisfies the eigenvalue condition} \}$ is open, dense, and of full Lebesgue measure in $GL_d(\R)^N$. In fact, the complement of the set is a finite union of $(d^2N-1)$-dimensional algebraic varieties.
\end{proposition}

\begin{proof}
  The claim basically follows from the first part of the proof of \cite[Corollary 2.7]{JJLS2014}. However, because of Remark \ref{rem:inaccuracy} and since the proof of \cite[Corollary 2.7]{JJLS2014} omits some of the details, we give a full proof for the convenience of the reader. We may clearly assume that $N=2$. Observe that the complement of $\{ \mathsf{A} \in GL_d(\R)^2 : \mathsf{A} \text{ satisfies the condition \ref{E1}} \}$ is
  \begin{equation*}
    \bigcup_{k=1}^d \bigcup_{(i_1,\ldots,i_k) \ne (j_1,\ldots,j_k)} \{ (A_1,A_2) \in GL_d(\R)^2 : \lambda_{i_1} \cdots \lambda_{i_k} = \lambda_{j_1} \cdots \lambda_{j_k} \text{ or } \lambda_{i_1}' \cdots \lambda_{i_k}' = \lambda_{j_1}' \cdots \lambda_{j_k}' \},
  \end{equation*}
  where $\lambda_1,\ldots,\lambda_d$ and $\lambda_1',\ldots,\lambda_d'$ are the eigenvalues of the matrices $A_1$ and $A_2$, respectively. To show that this complement is a finite union of $(2d^2-1)$-dimensional algebraic varieties it clearly suffices to show that $C = \{ (A_1,A_2) \in GL_d(\R)^2 : \lambda_{i_1} \cdots \lambda_{i_k} = \lambda_{j_1} \cdots \lambda_{j_k} \}$ is a $(2d^2-1)$-dimensional algebraic variety. Denoting
  \begin{equation*}
    g(\lambda) = \prod_{1 \le i_1 < \cdots < i_k \le d} (\lambda - \lambda_{i_1}\cdots\lambda_{i_k}),
  \end{equation*}
  we see that
  \begin{align*}
    C &= \{ \mathsf{A} \in GL_d(\R)^2 : \text{the polynomial $g$ has a multiple root} \} \\
    &= \{ \mathsf{A} \in GL_d(\R)^2 : \text{the discriminant of $g$ is $0$} \}.
  \end{align*}
  Note that the coefficients of $g$ can be expressed by the entries of $A_1$. As the discriminant of a polynomial is a symmetric function in the roots, it can be expressed in terms of the coefficients of the polynomial. Therefore, the discriminant of $g$ is a polynomial of the entries of $A_1$ and $C$ is a $(2d^2-1)$-dimensional algebraic variety.
  
  Let us next consider the condition \ref{E2}. Observe first that the complement of $\{ X \in GL_d(\R) : \text{all the minors of $X$ are nonzero} \}$ is
  \begin{align*}
    \mathcal{A} &= \bigcup_{k=1}^d \{ X \in GL_d(\R) : \text{there exists a zero minor of order $k$} \} \\
    &= \bigcup_{k=1}^d \bigcup_{\atop{1 \le i_1 < \cdots < i_k \le d}{1 \le j_1 < \cdots < j_k \le d}} \{ X \in GL_d(\R) : \det(X_{j_1,\ldots,j_k}^{i_1,\ldots,i_k}) = 0 \},
  \end{align*}
  where $X_{j_1,\ldots,j_k}^{i_1,\ldots,i_k}$ is the submatrix of $X$ corresponding to rows $i_1,\ldots,i_k$ and columns $j_1,\ldots,j_k$. Therefore, if $\mathsf{A}=(A_1,A_2) \in GL_d(\R)^2$ does not satisfy the condition \ref{E2}, then the matrix $X \in GL_d(\R)^2$ changing the eigenbases of $A_1$ and $A_2$ is contained in $\mathcal{A}$. Since the elements of $X$ are determined from $\mathsf{A}$ by some linear equations and the defining property of $X$ is a polynomial equation, we see that the complement of $\{ \mathsf{A} \in GL_d(\R)^2 : \mathsf{A} \text{ satisfies the condition \ref{E2}} \}$ is a finite union of $(2d^2-1)$-dimensional algebraic varieties.
  
  We have now finished the proof since the second claim trivially implies the first claim.
\end{proof}

\begin{remark} \label{rem:inaccuracy}
  There is a small inaccuracy in the statement of \cite[Corollary 2.7]{JJLS2014}. The corollary claims that a certain family of matrices satisfying the condition $C(s)$ is open. The proof, however, only verifies that this family contains an open set.
\end{remark}

\begin{proposition} \label{thm:eigenvalue-irred}
  If $\mathsf{A} \in GL_d(\R)^N$ satisfies the eigenvalue condition, then $\mathsf{A}$ is $s$-irreducible for all $0 \le s \le d$. In particular, for every $0 \le s \le d$ there exists a unique $\fii^s$-equilibrium state $\mu$ of $\mathsf{A}$ and it satisfies the following condition: there exists $C \ge 1$ depending only on $\mathsf{A}$ and $s$ such that
  \begin{equation*}
    C^{-1}e^{-nP_{\mathsf{A}}(\fii^s)}\fii^s(A_\iii) \le \mu([\iii]) \le Ce^{-nP_{\mathsf{A}}(\fii^s)}\fii^s(A_\iii)
  \end{equation*}
  for all $\iii \in \Sigma_*$.
\end{proposition}

\begin{proof}
  Since $(A_i,A_j)$ satisfies the eigenvalue condition for some $i,j \in \{ 1,\ldots,N \}$ with $i \ne j$ it follows from \cite[Theorem 2.6]{JJLS2014} (and the latter part of the proof of \cite[Corollary 2.7]{JJLS2014} which deals with the case of complex eigenvalues) that the family $\{ A_\kkk : \kkk \in \bigcup_{k=1}^{d'} \{ i,j \}^k \}$, where $d' = 2(\max\{ \binom{d}{m} : m \in \{ 0,\ldots,d \} \})^2$, satisfies the condition $C(s)$ for all $0 \le s \le d$. By Lemma \ref{thm:CS-cond}, we see that $(A_i,A_j)$, and consequently also $\mathsf{A}$, is $s$-irreducible for all $0 \le s \le d$. The second claim follows now immediately from Theorem \ref{thm:unique-eq-new}.
\end{proof}

\begin{proof}[Proof of Theorem \ref{thm:schmeling-generic}]
  It follows immediately from Propositions \ref{thm:eigenvalue-cond} and \ref{thm:eigenvalue-irred} that the set
  $\{ \mathsf{A} \in GL_d(\R)^N :$ there exists a unique $\fii^s$-equilibrium state of $\mathsf{A}$ and it is fully supported$\}$
  contains a set in $GL_d(\R)^N$ which is open, dense, and of full Lebesgue measure. Furthermore, the complement of the set is contained in a finite union of $(d^2N-1)$-dimensional algebraic varieties. Therefore, for any $\mathsf{A}$ in this set, Proposition \ref{thm:pressure-drop} implies that $\dimaff(\mathsf{A}') < \dimaff(\mathsf{A})$. The result of Falconer \cite[Theorem 5.3]{Falconer1988} finishes the proof.
\end{proof}

\begin{ack}
  The authors thank Pablo Shmerkin for discussions related to the topic of the paper and Ian Morris for pointing out Remark \ref{rem:morris}.
\end{ack}

% \bibliographystyle{abbrv}
% \bibliography{equilibrium}

\end{document}